\documentclass[11pt]{article}
\usepackage{amssymb,amsmath,amsthm, bbm}
\usepackage{enumerate}
\usepackage[dvipsnames]{xcolor}
\usepackage{pgf,tikz}
\usepackage{mathrsfs}
\usetikzlibrary{arrows}
\usepackage[unicode=true]{hyperref}
\hypersetup{
	pdftitle={k-planar crossing number of random regular graphs},
	pdfauthor={long list},
	colorlinks
}

\definecolor{darkgreen}{rgb}{0,0.6,0}

\let\oldmarginpar\marginpar
\renewcommand\marginpar[1]{\-\oldmarginpar[\raggedleft\footnotesize #1]%
{\raggedright\footnotesize #1}}

\newtheorem{theorem}{Theorem}

\newtheorem{lemma}[theorem]{Lemma}

\numberwithin{equation}{section}

\newcommand{\cro}{\textsc{cr}}

\newcommand{\old}[1]{{}}

\title{$k$-planar Crossing Number of Random Graphs and Random Regular Graphs}

\author{John~Asplund\footnotemark[1] \and Thao~Do\footnotemark[2] \and Arran~Hamm\footnotemark[3] \and L\'aszl\'o~Sz\'ekely \footnotemark[4] \and Libby~Taylor\footnotemark[5] \and Zhiyu~Wang\footnotemark[4]}

    \makeatletter
    \let\@fnsymbol\@arabic
    \makeatother

\newcommand{\vanish}[1]{}

\usepackage{verbatim}

\begin{document}

\footnotetext[1]{
Dalton State College, Department of Technology and Mathematics,
jasplund@daltonstate.edu
}
\footnotetext[2]{
Massachusetts Institute of Technology, Department of Mathematics, thaodo@mit.edu
}
\footnotetext[3]{
Winthrop University, Department of Mathematics, hamma@winthrop.edu
}
\footnotetext[4]{
University of South Carolina, Department of Mathematics, szekely@math.sc.edu and zhiyuw@math.sc.edu
}
\footnotetext[5]{
Georgia Institute of Technology, Department of Mathematics, libbytaylor@gatech.edu
}

\date{}
\maketitle

\begin{abstract}
We give an explicit extension of Spencer's result on the biplanar crossing number of the Erd\H{o}s-R\'enyi random graph $G(n,p)$. In particular, we show that the k-planar crossing number of $G(n,p)$ is almost surely $\Omega((n^2p)^2)$. Along the same lines, we prove that for any fixed $k$, the $k$-planar crossing number of various models of random $d$-regular graphs is $\Omega ((dn)^2)$ for $d > c_0$ for some constant $c_0=c_0(k)$. 
\end{abstract}

\section{Introduction}

Planar graphs have been heavily studied in the literature and their applications have sparked interdisciplinary work in a variety of fields, e.g., design problems for circuits, subways, and utility lines. The focus of this paper is a variation of the crossing number of a graph, which is itself a natural extension of planarity.
The \textit{crossing number} of a graph $G$, denoted $\cro(G)$, is the minimum number of edge crossings in a drawing of $G$ in the plane.
In particular, we will focus on the variation of crossing number known as the $k$-planar crossing number.  The $k$-planar crossing number of $G$, denoted $\cro_k (G)$, is defined as the minimum of $\cro(G_1)+\cdots + \cro(G_k)$ over all partitions of $G$ into $G_1\cup \cdots \cup G_k$.  The $k=2$ case is commonly referred to as the {\em biplanar crossing number}.

In this paper, we investigate the $k$-planar crossing number of two models of random graphs: Erd\H{o}s-R\'enyi random graphs and random $d$-regular graphs. Spencer, in \cite{spencer}, gave a lower bound on the biplanar crossing number of Erd\H{o}s-R\'enyi random graph $G(n,p)$.

\begin{theorem}\label{spenceThm1}
{\normalfont\cite{spencer}} 
There  are constants $c_0$ and $c_1$ such that for all $p\geq c_0/n$,  the biplanar crossing number $\cro_2(G)$, with $G \sim G(n,p)$, is with high probability at least $c_1(n^2p)^2$.
\end{theorem}


Spencer remarked that the methods used in Theorem \ref{spenceThm1} allow one to show that for all $k$, when $p\geq c_k/n$ for some $c_k$, $\cro_k(G)=\Omega((n^2p)^2)$ where $G \sim G(n,p)$. However, a few people in the community were unable to extend Spencer's proof, so in this paper, we give an explicit proof of the lower bound for the $k$-planar crossing number of $G(n,p)$ for arbitrary $k$. Throughout this paper, $o, O, \Omega$ are always for $n\rightarrow \infty$.

\begin{theorem}\label{k-planarErdosRenyi}
For all integers $k\geq 1$, there  are constants $c_0 = c_0(k)$ and $c_1 = c_1(k)$ such that for all $p\geq c_0/n$,  the $k$-planar crossing number of $G(n,p)$ is with high probability at least $c_1(n^2p)^2$.
\end{theorem}

Along similar lines, we investigate the k-planar crossing number of several models of random $d$-regular graphs in Section~\ref{proofofthmrandRegSpencerThm}. The key ingredients of the proof involve Friedman's results on Alon's second eigenvalue conjecture in \cite{friedman}.
In particular, we consider $\mathcal{G}_{n,d}, \mathcal{H}_{n,d}, \mathcal{I}_{n,d}, \mathcal{J}_{n,d}$ and some related models. Please refer to \cite{friedman} for the definitions of $\mathcal{G}_{n,d}, \mathcal{H}_{n,d}, \mathcal{I}_{n,d}, \mathcal{J}_{n,d}$. For two families of probability spaces, $(\Omega_n, \mathcal{F}_n, \mu_n)_{n=1,2,\cdots}$ and $(\Omega_n, \mathcal{F}_n, \nu_n)_{n=1,2,\cdots}$ over the same sets $\Omega_{n}$ and sigma-algebras $\mathcal{F}_n$, denote $\mu = \{\mu_n\}$ and $\nu = \{\nu_n\}$. We say \emph{$\mu$ dominates $\nu$} if for any family of measurable events $\{E_n\}$, $\mu_n(E_n) \to 0$ implies $\nu_n(E_n) \to 0$. We say that $\mu$ and $\nu$ are \emph{contiguous} if $\mu$ dominates $\nu$ and $\nu$ dominates $\mu$. Following Friedman's notation, let $\mathcal{L}_n$ be a family of probability spaces of d-regular graphs on $n$ vertices that is dominated by $\mathcal{G}_{n,d}, \mathcal{H}_{n,d}, \mathcal{I}_{n,d}$ or $\mathcal{J}_{n,d}$.

Along similar lines as the proof of Theorem \ref{k-planarErdosRenyi}, we prove that the $k$-planar crossing number of the random $d$-regular graph $G$ in $\mathcal{L}_n$ is $\Omega(n^2d^2)$, where $n$ is the number of vertices of $G$ and $d$ is the degree of regularity.

\begin{theorem}\label{randRegSpencerThm}
For all integers $k\geq 1$, there  are constants $c_0 = c_0(k)$ and $c_1 = c_1(k)$ such that for all $d\geq c_0$,  the $k$-planar crossing number $\cro_k(G)$, with $G$ in $\mathcal{L}_n$, is with high probability at least $c_1(n^2d^2)$.
\end{theorem}

The proof of Theorem \ref{randRegSpencerThm} hinges on the following 
 result on the edge densities of random $d$-regular graphs.

\old{
\begin{theorem}\label{newthm3}
There is a constant $c_0$ such that for any $d \geq c_0$ the random regular graph $G \sim \mathcal{G}_{n,d}$ has the following property with high probability: For every pair of disjoint sets $X,Y \subseteq V(G)$, each of size $\frac{n}{6}$, there are at least $\frac{dn}{72}$ edges $\{x,y\} \in E(G)$ with $x \in X, y\in Y$.
\end{theorem}
}

\begin{theorem}\label{newthm3}
For every $k$, there is a constant $c_0(k)$ such that for $d \geq c_0(k)$ the random $d$-regular graph $G$ in $\mathcal{L}_n$ has the following property with high probability: For every pair of disjoint sets $X,Y \subseteq V(G)$, each of size at least $\frac{n}{6\cdot 3^{k-2}}$, there are at least $\frac{dn}{2(6\cdot 3^{k-2})^2}$ edges $\{x,y\} \in G$ with $x \in X, y\in Y$.
\end{theorem}
\medskip 

We will also need to use the notion of bisection width, a key tool used to set lower bounds
for crossing numbers.  The \textit{$1/3$-$2/3$ bisection width} of $G$, denoted $b(G)$, is defined as
\[
b(G)=\min_{\substack{V_1\sqcup V_2=V \\ |V_i|\geq n/3}}\{e(V_1,V_2)\},
\]
where $e(V_1,V_2)$ is number of edges between $V_1$ and $V_2$.

The bisection width can intuitively be thought of as the minimum number of edges of $G$ which must be removed in order to disconnect $G$ into two connected components of roughly equal size.
An optimal $1/3$-$2/3$ bisection is a partition realizing the $1/3$-$2/3$ bisection width.
This parameter on $G$ is used in the following theorem to give a lower bound on the crossing number.
\begin{theorem}\label{bisection}
{\normalfont\cite{PSS}} Let $G$ be a graph of $n$ vertices, whose degrees are $d_1,d_2,\ldots,d_n$. Then
\[b(G)\leq 10\sqrt{\cro(G)}+2\sqrt{\sum_{i=1}^nd_i^2}.\]
\end{theorem}

\section{$k$-planar crossing number of $G(n,p)$}\label{kplanarRandomGraph}

The proof goes by a sequence of lemmas.

\begin{lemma}\label{lemma1}
{\normalfont\cite{spencer}} Let $X$ be an $m$-element set and let $X_{11}\cup X_{12}$ and $X_{21}\cup X_{22}$ be any two bipartitions of $X$ such that each of $X_{11},X_{12},X_{21}$ and $X_{22}$ has size at least $m/3$.  Then there exist two subsets $Y_1$ and $Y_2$ of $X$, each of size at least $m/6$, which lie on different sides of both bipartitions.  That is, either $Y_1\subseteq X_{11} \cap X_{21}, Y_2\subseteq X_{12} \cap X_{22}$ or $Y_1\subseteq X_{11} \cap X_{22}$, $Y_2\subseteq X_{12} \cap X_{21}$.
\end{lemma}

The following lemma is a slight variation of Theorem~3 in \cite{spencer}.

\begin{lemma}\label{lemma2}
When $p=\Omega(1/n)$, the random graph $G(n,p)$ satisfies the following property w.h.p.:
for any fixed $k$ and for every pair of disjoint vertex sets $X,Y$, each of size at least $t=n/(6\cdot 3^{k-2})$, there are at least $\frac{1}{2}\binom{t}{2} p$ 
edges $\{x,y\} \in G$ with $x \in X, y\in Y$.
\end{lemma}


\begin{proof}
Large deviation inequalities (see, e.g. Theorem A.1.13 of the Appendix of \cite{AS})  provides that $M$ independent trials, each  with probability $p$, provides fewer than $Mp/2$ successes with probability at most $e^{-Mp/8}$. Hence the number of edges between our sets is under $(1/2)pn^2/(6\cdot 3^{k-2})^2$ with probability at most $e^{- pn^2/(288\cdot 9^{k-2})    }$. This bound, combined with an upper bound 
$4^n$ to the number of choices of $X,Y$, establishes the lemma.
\end{proof}

It is clear that the expected degree of each vertex is $(n-1)p$. 
The following bound on the maximum degree of $G(n,p)$ follows from the Chernoff bound for independent Bernoulli random variables.

\begin{lemma}\label{degreebound}
In $G(n,p)$ with $p>m/n$ for some constant $m$, with high probability 
the maximum degree in $G(n,p)$ is at most $(1+\log n)np$. 
\end{lemma}
\begin{proof}
If $X$ is the sum of $n-1$ independent Bernoulli random variables, which take value $1$ with probability $p$, then for every $\delta>0$ the Chernoff bound gives that 
$$\mathbb{P}[X>(1+\delta)(n-1)p]\leq \Biggl(\frac{e^\delta}{{(1+\delta)}^{1+\delta}}\Biggl)^{(n-1)p}.$$
Applying this bound with $\delta=\log n$ gives an upper bound of $n^{-\log \log n +1}$ on the probability that a vertex has degree above $pn(1+\log n)$, if $p>2/n$.  The union bound then gives that the expected number of vertices with degree above $pn(1+\log n)$ is $n^{-\log \log n +2}=o(1)$ which proves the assertion.
\end{proof}

This condition on maximum degree is necessary in order to use Theorem \ref{bisection} to provide the lower bound on the crossing number.  In order for Theorem \ref{bisection} to give $\cro_k(G)=\Omega((n^2p)^2)$, it must be the case that $\sqrt{\sum d_i^2}=o(b(G))$.  If $\Delta$ denotes the maximum degree of a vertex in $G$ 
, then it is sufficient that $\Delta\sqrt{n}=o(n^2p)$, which is satisfied when $\Delta\leq (1+\log n)np$ and $p>m/n$ for a sufficiently large constant $m$.

\begin{proof}[Proof of Theorem~{\normalfont\ref{k-planarErdosRenyi}}]
Let $G_1\cup G_2\cup \cdots \cup G_k$ be the 
partition of the edges of our sample $G$ of $G(n,p)$ that realizes
the  $k$-planar crossing number of $G$. We assume without loss of 
generality that the sample satisfies the requirements in Lemmas~\ref{lemma2} and \ref{degreebound}.

Consider now the optimal $1/3$-$2/3$ bisections of $G_1$ and $G_2$. By Lemma~\ref{lemma1}, we find two disjoint sets, each of the same size, at least $n/6$, which are separated from each other by both optimal $1/3$- $2/3$ bisections. Call the union of these 
two sets $Y_2$, and observe $|Y_2|\geq n/3$.

Let $G_3|_{Y_2}$ denote the restriction of $G_3$ to the vertices of $Y_2$ and consider now the optimal $1/3$-$2/3$ bisection of $G_3|_{Y_2}$, 
and an equipartition of $Y_2$ similar to that which we used to define $Y_2$.
Lemma~\ref{lemma1} applies again, resulting in  two disjoint, equal sized subsets of $Y_2$, of size at least $n/(3\cdot 6)$,   which are subsets of different
sides of all three partitions we have considered so far. Call the union of 
these two sets $Y_3$, and observe $|Y_3|\geq n/3^2$. 

If for some $3\leq i\leq k-1 $ the set $Y_i$ is already defined, 
consider now the optimal $1/3$-$2/3$ bisection of the graph 
$G_{i+1}|_{Y_i}$, and the partition 
of $Y_i$ from which we defined 
$Y_i$. Lemma~\ref{lemma1} applies again, resulting in  two equal sized disjoint 
 subsets of $Y_i$ of size at least 
$n/(3^{i-1}\cdot 6)$, 
which 
are subsets of different
sides of all $i+1$ partitions considered so far. Call the union of 
these two sets $Y_{i+1}$, and observe $|Y_{i+1}|\geq n/ 3^{i}$.

Let $A$ and $B$ denote the two disjoint, equal sized sets, of size at least $n/(3^{k-2}\cdot 6)$, whose union defined $Y_k$.
The following inequality follows from our construction:

$$e_G(A,B)\leq b(G_k|_{Y_{k-1}})+b(G_{k-1}|_{Y_{k-2}})+\ldots 
+b(G_3|_{Y_{2}})+b(G_2)+b(G_1).$$

Observe that $e_G(A,B)$ is large by Lemma~\ref{lemma2}. Therefore, at least one of 
the $1/3$-$2/3$ bisection widths of the $k$ graphs on the right-hand side must be large; that is, at least $\frac{n^2p}{k\cdot2(6\cdot 3^{k-2})^2}$.  (Note that the additional factor of $1/k$ comes from the fact that there are $k$ total summands on the right hand side.) 
In particular, it is at least a constant fraction of $n^2p$.  By Theorem ~\ref{bisection}, the bisection width of this graph is large enough to prove Theorem \ref{spenceThm1}.

\end{proof}

\section{Proof of Theorem~\ref{newthm3}}\label{proofnewthm3}

For any $X,Y\subseteq V(G)$, let $E(X,Y)$ be the set of edges with one vertex in $X$ and the other in $Y$ and denote the order of $E(X,Y)$ by $e(X,Y)$.

The following variant of \emph{Expander Mixing Lemma} is a slight extension by Beigel, Margulis and Spielman \cite{BMS} of a bound originally proven by Alon and Chung\cite{alon1988explicit}.

\begin{theorem}\label{expandermixinglemma}
{\normalfont\cite{BMS}}
Let $G$ be a $d$-regular graph such that every eigenvalue except the largest has absolute value at most $\mu$.  Let $X,Y\subset V$ have sizes $\alpha n$ and $\beta n$, respectively. Then 
\begin{align*}
|e(X,Y)-\alpha\beta dn|\le \mu n \sqrt{(\alpha-\alpha^2)(\beta-\beta^2)}
\end{align*}

\end{theorem}

The following theorem of Friedman \cite{friedman} gives a bound on $\mu$ for random regular graphs.

\begin{theorem}\label{eigenvaluebound}
{\normalfont\cite{friedman}}
Fix a real $\varepsilon>0$ and a positive integer $d\ge 2$.  Let $\lambda_i$ denote the $i^{\rm th}$ eigenvalue of the adjacency matrix of $G$.  Then there exists a constant $c$ such that for a random $d$-regular graph $G$ in $\mathcal{L}_n$, we have with probability $1-o(1)$ (as $n\to \infty$) that for all $i>1$, 
\begin{align*}
|\lambda_i(G)|\le 2\sqrt{d-1}+\varepsilon
\end{align*}

\end{theorem}

\begin{proof}[Proof of Theorem~{\normalfont\ref{newthm3}}]

Theorem \ref{eigenvaluebound} gives that the $\mu$ in Theorem \ref{expandermixinglemma} is at most $2\sqrt{d-1}+\varepsilon$ with high probability when $n$ is large.  As in Theorem \ref{newthm3}, let $\alpha=\beta=\frac{1}{6\cdot 3^{k-2}}$.  This gives that with high probability,
\[\bigg|e(X,Y)-\frac{dn}{(6\cdot 3^{k-2})^2}\bigg|\le 2n\sqrt{d-1}\sqrt{\bigg(\frac{1}{6\cdot 3^{k-2}}-\bigg(\frac{1}{6\cdot 3^{k-2}}\bigg)^2 \bigg)^2}
\]
The radical on the right hand side is a constant for fixed $k$.  Let twice this constant be defined as $c_k$.  Then with high probability, $e(X,Y)$ differs from $\frac{dn}{(6\cdot 3^{k-2})^2}$ by at most $c_k n\sqrt{d-1}$, where $c_k<1$ for all $k\ge 2$.  Straightforward computation shows that when $d\ge (4\cdot 6 \cdot3^{k-2})^2=:c_0(k)$, it will be the case that the right hand side of the inequality at most $\frac{dn}{2(6\cdot 3^{k-2})^2}$ which completes the proof.
\end{proof}


\section{Proof of Theorem \ref{randRegSpencerThm}}\label{proofofthmrandRegSpencerThm}

In this section we will prove Theorem \ref{randRegSpencerThm}: Let $G$ be a random $d$-regular graph in $\mathcal{L}_n$. Then
\[
\cro_k(G)=\Theta((dn)^2).
\]

\begin{proof}[Proof of Theorem~{\normalfont\ref{randRegSpencerThm}}]
Looking closely at the proof of Theorem~\ref{spenceThm1}, we realize any graph $G$ that satisfies the conditions in Theorem~\ref{newthm3} (about edge density between two sufficiently large vertex sets) and Lemma~\ref{degreebound} (about maximum degree) will have large $k$-planar crossing numbers. In our case, we can replace $p$ by $d/n$, the density of our $d$-regular graph $G$. Clearly $G$ has maximum degree $d$ which satisfies Lemma~\ref{degreebound}. 
Therefore Theorem~\ref{randRegSpencerThm} holds.
\end{proof}

\section*{Acknowledgements}

\thanks{This material is based upon work that started at the Mathematics Research Communities workshop ``Beyond Planarity: Crossing Numbers of Graphs'', organized by the American Mathematical Society, with the 
support of the National Science Foundation under Grant Number DMS 1641020.}

\thanks{The fourth and the sixth authors were also supported in part by NSF DMS grants 1300547 and 1600811.  We would like to thank the  organizers of the workshop for creating a stimulating and productive environment, and Lutz Warnke for many helpful comments.}

\bibliographystyle{amsplain}
\bibliography{vdec}

\end{document}